\documentclass[11pt]{amsart}
\usepackage[T1]{fontenc}
\usepackage{lmodern}
\usepackage{amsmath,amssymb,amsthm,mathtools}
\usepackage[margin=1in]{geometry}
\usepackage[expansion=false]{microtype}
\usepackage{tikz-cd}
\usepackage[colorlinks=true,linkcolor=blue,citecolor=blue,urlcolor=blue]{hyperref}
\hypersetup{pdftitle={Commutators of compact operators}}
\numberwithin{equation}{section}
\newtheorem{theorem}{Theorem}[section]
\newtheorem{proposition}[theorem]{Proposition}
\newtheorem{lemma}[theorem]{Lemma}
\newtheorem{corollary}[theorem]{Corollary}
\theoremstyle{definition}
\newtheorem{definition}[theorem]{Definition}

\DeclareMathOperator{\Tr}{Tr}
\DeclareMathOperator{\ran}{ran}
\DeclareMathOperator{\rank}{rank}
\DeclareMathOperator{\spanop}{span}
\newcommand{\HH}{\mathcal H}
\newcommand{\BB}{\mathcal B}
\newcommand{\KK}{\mathcal K}
\newcommand{\SSp}{\mathcal S}
\newcommand{\C}{\mathbb C}
\newcommand{\N}{\mathbb N}
\newcommand{\ip}[2]{\langle #1,#2\rangle}
\allowdisplaybreaks[1]
\title[]{Every compact operator is a commutator of compact operators}
\author{Zhichao Liu}
\address{School of Mathematical Sciences,
Dalian University of Technology,
Dalian, {\rm 116024}, China }
\email{lzc.12@outlook.com}

\date{}

\begin{document}
\begin{abstract}
We prove that every compact operator $T$ on a separable infinite-dimensional complex Hilbert space is a commutator of two compact operators, thereby answering a long-standing open question of Pearcy and Topping. Moreover, the compact factors $A$ and $B$ can be chosen such that $[A, B] = T$ and $\max\{\Vert{}A\Vert{},\Vert{}B\Vert{}\}\leq c\Vert{}T\Vert{}^{1/2}$ for a universal constant $c$.
\end{abstract}
\maketitle

\section{Introduction}

Let $\HH$ be a separable infinite-dimensional complex Hilbert space,
and let $\BB(\HH)$ and $\KK(\HH)$ denote the bounded and compact
operators on $\HH$, respectively. A \emph{commutator} is an operator
of the form $[A,B]=AB-BA$. The question of which operators arise as commutators depends crucially on the class to which the factors are restricted. Brown and Pearcy \cite{BrownPearcy} characterized
the commutators in $\BB(\HH)$: the only excluded operators are
$\lambda I+K$, where $\lambda\ne0$ and $K$ is compact. In particular,
every compact operator is a commutator of bounded operators.

In 1971, Pearcy and Topping \cite{PT} asked whether the factors can
both be chosen compact. Thus their question is whether
\[
\KK(\HH)=\{[A,B]:A,B\in\KK(\HH)\}.
\]
The structure of the linear spans of commutators in
operator ideals has a separate and extensive theory; see
Dykema, Figiel, Weiss, and Wodzicki \cite{DFWW}. 

An important step was due to Anderson \cite{Anderson}, who showed that a rank-one projection as a commutator of two compact operators.
His work also gives the conclusion for compact operators that vanish on an infinite-dimensional reducing subspace, and hence for all
finite-rank operators. The rank-one example shows why the ordinary trace cannot be used as an obstruction: a trace-class commutator of
compact operators need not have trace zero when its two products are
not individually trace class.

Motivated by Anderson's result, subsequent work sought strictly positive compact operators
that are commutators of compact operators \cite{W}.
Davidson, Marcoux, and Radjavi constructed such examples
in an unpublished manuscript \cite{DMR}.
Belti{\c{t}}{\u{a}}, Patnaik, and Weiss independently obtained
families of strictly positive examples by modifying
Anderson's construction \cite[Theorem~3.1]{BPW};
see also \cite{P}. 
Dykema and Krishnaswamy-Usha
\cite{DKU} proved that every nilpotent element of an operator ideal
is a single commutator in a suitable power of that ideal; in particular,
every nilpotent compact operator is a commutator of compact operators.
More recently, Loreaux, Patnaik, Petrovi\'c, and Weiss \cite{LPPW}
extended the classes accessible through Anderson's method and
identified restrictions on that method. They also use block tridiagonal forms to reformulate the general problem and examine the restrictions of compactness. Inspired by this block-matrix approach, our construction relies on uniform finite-dimensional commutator estimates to bound the norms of the factors and ensure their compactness.

A different method comes from the study of diagonals. Fan
\cite{Fan} characterized operators having a zero-diagonal in an
orthonormal basis. In particular, a trace-class operator of trace
zero has a zero-diagonal. Fan and Fong \cite{FF} proved that every
compact operator outside the trace class is similar to an operator
with zero diagonal. Reducing the operator to finite-dimensional blocks is a natural strategy, and it relies on controlling the matrix commutators as the block size grows. Recently, the bound established in \cite[Theorem~1.1]{SWZ} provides the uniform estimate needed to advance this approach.

In this paper, we give an affirmative answer to the 1971 Pearcy-Topping
question and obtain a uniform estimate for the compact factors.

Throughout, fix a universal constant $\kappa\geq 4$ such that Lemma~\ref{lem:matrix} holds.

\begin{theorem}\label{thm:main}
For every $T\in\KK(\HH)$, there exist $A,B\in\KK(\HH)$ such that
\[
T=[A,B],\qquad
\max\{\|A\|,\|B\|\}\leq 50\sqrt{\kappa\|T\|}.
\]
\end{theorem}


In this paper, we distinguish two cases according to whether the compact operator belongs to the trace class. For compact operators outside the trace class, we show that the similarity in the Fan--Fong reduction can be chosen arbitrarily close to the identity. For trace-class operators, we use a similarity
to concentrate the trace in a finite-rank diagonal block.
We then combine Anderson's construction on this block with
the zero-diagonal construction on the remaining summand.
The remainder of the paper is organized as follows. Section~2 introduces notation and establishes preliminary estimates. Section~3 is devoted to the relevant compression and techniques. Finally, Section~4 presents the quantitative  reduction and the proof of Theorem~\ref{thm:main}.

\section{Preliminaries and the zero-diagonal construction}

In this paper, all Hilbert spaces are complex, and inner products are linear in the first variable.

\begin{definition}\label{def:operators}
For Hilbert spaces $\HH_1,\HH_2$, let $\BB(\HH_2,\HH_1)$ and
$\KK(\HH_2,\HH_1)$ denote the bounded and compact operators from
$\HH_2$ to $\HH_1$, respectively. When the two spaces agree, we write
$\BB(\HH)$ and $\KK(\HH)$. Denote by $\|T\|$ the operator
norm of $T$ and by $\sigma(T)$  the spectrum of $T$.
\end{definition}

\begin{definition}
An operator $T\in \BB(\HH)$ has \emph{a zero-diagonal} if there exists an orthonormal basis
$(e_j)$ for $\HH$ such that $\ip{Te_j}{e_j}=0$ for all $j$.
\end{definition}

\begin{definition}\label{def:trace}
An operator $T\in\BB(\HH)$ is \emph{trace-class} if
$\sum_{j=1}^{\infty}\ip{|T|e_j}{e_j}<\infty$ for an orthonormal basis $(e_j)$,
where $|T|=(T^*T)^{1/2}$. The trace class is denoted by
$\SSp_1(\HH)$, and its norm is
$\|T\|_1=\Tr(|T|)$. For $T\in\SSp_1(\HH)$,
\[
\Tr(T)=\sum_{j=1}^{\infty}\ip{Te_j}{e_j}.
\]
The series is absolutely convergent and independent of the basis.
For an operator $T$ on a finite-dimensional space, we write $\Tr(T)$ for the unnormalized trace.
\end{definition}
Note that for any bounded operators $A,B$ and trace-class $T$,
$$
\|ATB\|_1\leq\|A\|\,\|T\|_1\,\|B\|,\quad \Tr(AT)=\Tr(TA).$$
We use the following form of the matrix theorem of Shen, Wang, and Zhi.
Its essential point is that the same constant works in every matrix
dimension. 
The equal bounds follow
from reciprocal rescaling: for nonzero factors, replace $U,V$ by
$cU,c^{-1}V$ with $c=(\|V\|/\|U\|)^{1/2}$.

\begin{lemma}[{\cite[Theorem~1.1]{SWZ}}]
\label{lem:matrix}
There is a constant $\kappa\geq1$ such
that every trace-zero matrix $H$ has a representation
\[
H=[U,V],\qquad
\max\{\|U\|,\|V\|\}\leq\sqrt{\kappa\|H\|}.
\]
\end{lemma}


We shall also use the standard  fact that, if finite-rank
orthogonal projections $P_n$ increase strongly to $I$, then
$\|(I-P_n)K\|\to0$ and $\|K(I-P_n)\|\to0$ for every compact $K$.

The next estimate is an elementary case of Rosenblum's
operator-equation theorem \cite{Rosenblum}; see also \cite[Lemma~2.4]{SWZ}. We retain its short proof.

\begin{lemma}\label{lem:sylvester}
Let $U\in\BB(\HH_1)$, $V\in\BB(\HH_2)$, and $z,w\in\C$ satisfy
$|z-w|>\|U\|+\|V\|$. For every $E\in\BB(\HH_2,\HH_1)$, the equation
\[
(zI+U)X-X(wI+V)=E
\]
has a unique solution $X\in\BB(\HH_2,\HH_1)$, with
\[
\|X\|\leq\frac{\|E\|}{|z-w|-\|U\|-\|V\|}.
\]
\end{lemma}

\begin{proof}
On $\BB(\HH_2,\HH_1)$, the map $\Psi(X)=UX-XV$ has norm at most
$\|U\|+\|V\|<|z-w|$. Hence $(z-w)I+\Psi$ is invertible and
\[
\|((z-w)I+\Psi)^{-1}\|
\leq\frac1{|z-w|}\sum_{j=0}^{\infty}
       \left(\frac{\|\Psi\|}{|z-w|}\right)^j
\leq\frac1{|z-w|-\|U\|-\|V\|}.
\]
Applying this inverse to $E$ proves the assertion.

\end{proof}

We also use a resolvent form of this argument \cite{Rosenblum}. By bounding the resolvent rather than using norm separation (as in Lemma~\ref{lem:sylvester}), the following Neumann series directly provides the required estimates.

\begin{lemma}\label{lem:resolvent}
Let $A_0\in\BB(\HH_0)$, $\delta\notin\sigma(A_0)$, and
$r=\|(A_0-\delta I)^{-1}\|$. Suppose that $U\in\BB(\mathcal E)$
and $r\|U\|\leq1/2$. For
$E\in\BB(\mathcal E,\HH_0)$ and $F\in\BB(\HH_0,\mathcal E)$,
the equations
\[
A_0X-X(\delta I+U)=E,\qquad
(\delta I+U)Y-YA_0=F
\]
have unique solutions satisfying
$\|X\|\leq2r\|E\|$ and $\|Y\|\leq2r\|F\|$.
\end{lemma}

\begin{proof}
Let $W=(A_0-\delta I)^{-1}$. Multiplying the first equation on the left by $W$ and the second on the right by $W$ yields
$$X=WE+WXU,\quad Y=UYW-FW.$$
Define
\[
X=\sum_{j=0}^{\infty}W^{j+1}EU^j,
\qquad
Y=-\sum_{j=0}^{\infty}U^jFW^{j+1}.
\]
The estimates
\[
\|W^{j+1}EU^j\|\leq r\|E\|\,2^{-j},\qquad
\|U^jFW^{j+1}\|\leq r\|F\|\,2^{-j}
\]
show that both series converge absolutely in the operator norm. Direct substitution into the equations confirms that their sums are indeed the desired unique solutions. 
Now the geometric bounds yield
\[
\|X\|\leq\frac{r\|E\|}{1-r\|U\|}\leq2r\|E\|,
\qquad
\|Y\|\leq\frac{r\|F\|}{1-r\|U\|}\leq2r\|F\|.
\]

\end{proof}

The following block form  will control the
operators obtained from these equations. 

\begin{lemma}\label{lem:schur}
Let $(\mathcal E_n)_{n\geq1}$ be Hilbert spaces, let
$\mathcal E=\bigoplus_{n\geq1}\mathcal E_n$, and suppose that
$Z_{mn}\in\BB(\mathcal E_n,\mathcal E_m)$ for all $m,n\geq1$. Let $c>0$.
If
\[
\sup_{m}\sum_{n=1}^{\infty}\|Z_{mn}\|\leq c,
\qquad
\sup_{n}\sum_{m=1}^{\infty}\|Z_{mn}\|\leq c,
\]
then the block matrix
\[
Z=\begin{pmatrix}
Z_{11}&Z_{12}&Z_{13}&\cdots\\
Z_{21}&Z_{22}&Z_{23}&\cdots\\
Z_{31}&Z_{32}&Z_{33}&\cdots\\
\vdots&\vdots&\vdots&\ddots
\end{pmatrix}
\begin{array}{l}\mathcal E_1\\\mathcal E_2\\\mathcal E_3\\\vdots\end{array}
\]
defines a bounded operator on $\mathcal E$, and $\|Z\|\leq c$.
\end{lemma}

\begin{proof}
For a finitely supported vector $x=(x_n)$, by the Cauchy--Schwarz inequality,
\begin{align*}
\sum_{m=1}^{\infty}\left\|\sum_{n=1}^{\infty}Z_{mn}x_n\right\|^2
&\leq\sum_{m=1}^{\infty}\left(\sum_{n=1}^{\infty}\|Z_{mn}\|\right)
                 \left(\sum_{n=1}^{\infty}\|Z_{mn}\|\,\|x_n\|^2\right)\\
&\leq c\sum_{n=1}^{\infty}\|x_n\|^2\sum_{m=1}^{\infty}\|Z_{mn}\|
\leq c^2\sum_{n=1}^{\infty}\|x_n\|^2.
\end{align*}
The formula therefore defines an operator taking values in the Hilbert direct sum, which extends uniquely by continuity from the dense subspace of finitely supported vectors to a bounded operator satisfying the stated norm bound.

\end{proof}

\begin{proposition}\label{prop:zero-diagonal}
Suppose that $T\in\mathcal K(\mathcal H)$ has a zero-diagonal with respect to an orthonormal basis
$(e_j)$.
Then there exist $A,B\in\mathcal K(\mathcal H)$ such that
$$
T=[A,B]
\quad {\rm and}\quad
\|A\|\,\|B\|\leq4\kappa\|T\|.
$$
\end{proposition}

\begin{proof}
Assume $T$ is non-zero. Let $\delta_1>0$ satisfy $\delta_1^2\geq\kappa\|T\|$. Then we choose a sequence of positive numbers $\delta_n$ for $n\geq2$ such that
$$
\delta_{n+1}\leq\frac{\delta_n}{4},
\qquad n\geq1.
$$
By compactness of $T$, we can choose integers $0<d_1<d_2<\cdots$ such that the
projection $P_n$ onto $\spanop\{e_1,\ldots,e_{d_n}\}$ satisfies
$$
\max\{\|(I-P_n)T\|,\|T(I-P_n)\|\}
\leq \frac{\delta_{n+1}^2}{\kappa},\qquad n\geq1.
$$
Set $P_0=0$, $Q_n=P_n-P_{n-1}$ and $\mathcal E_n=Q_n\HH$. Then
$\HH=\bigoplus_{n\geq1}\mathcal E_n$. Write
\[
T_{mn}=Q_mT|_{\mathcal E_n}\in\BB(\mathcal E_n,\mathcal E_m).
\]
Since $T$ has a zero-diagonal, $T_{nn}$ has trace zero for all $n$. Furthermore,
$$
\|T_{mn}\|\leq\min \{\|Q_mT\|,\|TQ_n\|\} \leq\frac{\min\{\delta_m,\delta_n\}^2}{\kappa}.
$$
By Lemma~\ref{lem:matrix}, there exist matrices $U_n,V_n$ such that
\[
T_{nn}=[U_n,V_n],\qquad
\max\{\|U_n\|,\|V_n\|\}
\leq\sqrt{\kappa\|T_{nn}\|}\leq \delta_n.
\]

Let $A_n=2\delta_nI+U_n$ and $A={\rm diag}\{A_1,A_2,\cdots,A_n,\cdots\}$.
Then
$$
\|A\|\leq3\delta_1,\quad\|A-P_kAP_k\|\leq 3\delta_{k+1}.
$$

 Now we construct the matrix $B$.
Define $Z_{nn}=V_n$ for all $n$. If $m<n$, consider the operator equation
$$A_mX-XA_n=T_{mn}.$$
Since
$\delta_n\leq\delta_m/4$,
\[
|2\delta_m-2\delta_n|-\|U_m\|-\|U_n\|
\geq \frac{\delta_m}{4}.
\]
By Lemma~\ref{lem:sylvester}, the operator equation has a unique solution, which we denote by $Z_{mn}$.
For $n < m$, we define $Z_{mn}$ analogously by interchanging the indices.

Now we have the estimate
$$
\|Z_{mn}\|\leq \frac{\|T_{mn}\|}{|2\delta_m-2\delta_n|-\|U_m\|-\|U_n\|}\leq
4\frac{\min\{\delta_m,\delta_n\}^2}
                         {\max\{\delta_m,\delta_n\}},
\quad m\ne n.
$$

Let $B=(Z_{mn})$.  Thus the matrices  $A,B$ have
the form
\[
A=\begin{pmatrix}
A_1&0&0&\cdots\\0&A_2&0&\cdots\\0&0&A_3&\cdots\\
\vdots&\vdots&\vdots&\ddots
\end{pmatrix},\qquad
B=\begin{pmatrix}
V_1&Z_{12}&Z_{13}&\cdots\\
Z_{21}&V_2&Z_{23}&\cdots\\
Z_{31}&Z_{32}&V_3&\cdots\\
\vdots&\vdots&\vdots&\ddots
\end{pmatrix}.
\]

We now proceed to estimate the sum of the row norms. For $m\geq 2$,
\begin{align*}
  \sum_{n=1}^{\infty}\|Z_{mn}\| & \leq \sum_{n<m}\|Z_{mn}\|+ \|V_m\|+\sum_{n>m}\|Z_{mn}\|\\
   & \leq 4\sum_{n<m}\frac{\delta_m^2}{\delta_n}+\delta_m +4\sum_{n>m}\frac{\delta_n^2}{\delta_m}\\
   & \leq 4\delta_m\sum_{j=1}^{m-1}\frac{1}{4^j} +\delta_m +4\delta_m\sum_{j=1}^{\infty}\frac{1}{16^j} \\
   & < \frac{4}{3}\delta_m+\delta_m+\frac{1}{3}\delta_m < \delta_1.
\end{align*}
When $m=1$, then
$$
  \sum_{n=1}^{\infty}\|Z_{1n}\| \leq \|V_1\|+\sum_{n>1}\|Z_{1n}\|
  \leq \delta_1 +4\delta_1\sum_{j=1}^{\infty}\frac{1}{16^j}\leq\frac{4}{3}\delta_1.
$$
Hence,
$$
\sup_{m\geq 1} \sum_{n=1}^{\infty}\|Z_{mn}\| <\frac{4}{3}\delta_1 .
$$
Since the estimate is symmetric for $m,n$, we also have
$$
\sup_{n\geq 1} \sum_{m=1}^{\infty}\|Z_{mn}\| <\frac{4}{3}\delta_1.
$$
By Lemma~\ref{lem:schur}, we have
$$
\|B\|\leq \frac{4}{3}\delta_1 .
$$

The same estimates hold for $B-P_kBP_k$, we have
\[
\|B-P_kBP_k\|\leq\frac{4}{3}\delta_{k+1}.
\]
Since $P_kAP_k\to A$ and $P_kBP_k\to B$ in norm,  $A$ and $B$ are compact.

For $m,n\leq k$,
\[
\bigl([P_kAP_k,P_kBP_k]\bigr)_{mn}
=\begin{cases}
[U_n,V_n],&m=n,\\
A_mZ_{mn}-Z_{mn}A_n,&m\ne n.
\end{cases}
\]
This means that $[P_kAP_k,P_kBP_k]=P_kTP_k$. Finally,
\[
\|T-P_kTP_k\|
\leq\|(I-P_k)T\|+\|P_kT(I-P_k)\|
\leq \frac{2}{\kappa}\delta_{k+1}^2.
\]
Passing to the norm limit, we have
$$[A,B]=\lim_{k}[P_kAP_k,P_kBP_k]=\lim_{k} P_kTP_k=T.$$
Moreover,
$$
\|A\|\,\|B\|\leq 3\delta_1 \cdot \frac{4}{3}\delta_1=4 \delta_1^2.
$$
Taking $\delta_1=\sqrt{\kappa\|T\|}$ gives
$$
\|A\|\,\|B\|\leq4\kappa\|T\|.
$$
\end{proof}
\section{Operator Decomposition and Idempotent Compression}

We first construct a bounded idempotent whose range has zero
compression under $T$. The range and kernel will then give the
decomposition needed to separate the trace. 

\begin{lemma}\label{lem:compression}
Let $T\in\KK(\HH)$ have infinite rank, and let
$\mathcal M\subset\HH$ be finite-dimensional. There are
infinite-dimensional closed subspaces
$\mathcal L,\mathcal W\subset\mathcal M^\perp$ and a bounded
idempotent $Q$ such that
\[
T\mathcal L\subset\mathcal W,\qquad
\ran Q=\mathcal L,\qquad
\mathcal M,\mathcal W\subset\ker Q,\qquad
\|Q\|\leq2.
\]
In particular, $QTQ=0$, and both $\ran Q$ and $\ker Q$ are
infinite-dimensional.
\end{lemma}

\begin{proof}
Choose positive numbers $\eta_n$ with
$\sum_n\eta_n^2\leq1/100$.
We first construct an orthonormal family
$(x_1,y_1,x_2,y_2,\ldots)$ satisfying
\begin{align*}
&x_n,y_n,Tx_n,Ty_n\in\mathcal M^\perp,\\
&\langle Tx_n, Tx_j\rangle=0\,\,(n\ne j),\quad Tx_n\ne0,\\
&\|Ty_n\|\leq\eta_n\|Tx_n\|,\quad \langle y_n, Tx_j\rangle=0\,\,(n,j\geq1).
\end{align*}
Set $\mathcal F_1=\mathcal M+T^*\mathcal M$. Since the restriction of $T$ to $\mathcal F_1^\perp$ is nonzero, otherwise $T=TP_{\mathcal F_n}$ would have finite rank.
We can choose a unit vector $x_1\in\mathcal F_1^\perp$ with $Tx_1\ne0$. Set the space
\[
\mathcal G_1=\mathcal M+T^*\mathcal M+\spanop\{x_1,Tx_1\}.
\]
Then $\mathcal G_1$ is finite-dimensional. Since $T$ is compact, the sequence of images of any orthonormal basis for $\mathcal G_1^\perp$ under $T$ converges to zero in norm.
Thus we can choose a unit vector $y_1\in G_1^\perp$ with
$$\|Ty_1\|\leq\eta_1\|Tx_1\|.$$

Assume inductively that the pairs with indices strictly less than $n$ have already been constructed.
Set
\[
\mathcal F_n=\mathcal M+T^*\mathcal M+
\spanop\{x_j,y_j,T^*Tx_j,T^*y_j:j<n\}.
\]
Choose a unit vector $x_n\in\mathcal F_n^\perp$
with $Tx_n\ne0$. This preserves orthonormality and gives
\[
\ip{Tx_n}{Tx_j}=\ip{x_n}{T^*Tx_j}=0,\qquad
\ip{Tx_n}{y_j}=\ip{x_n}{T^*y_j}=0\quad(j<n).
\]
Now set
\[
\mathcal G_n=\mathcal M+T^*\mathcal M+
\spanop\{x_j,Tx_j:j\leq n\}+
\spanop\{y_j:j<n\}.
\]
Again by compactness, we can choose a unit vector $y_n\in\mathcal G_n^\perp$
with $$\|Ty_n\|\leq\eta_n\|Tx_n\|.$$
Clearly, $\langle y_n, Tx_j\rangle=0$ for $j \leq n$. For $j > n$, the subsequent choice of $x_j$ ensures that $\langle y_n, Tx_j\rangle=0$.

Finally, orthogonality to $\mathcal M+T^*\mathcal M$ puts
each chosen vector and its image in $\mathcal M^\perp$.
This completes the inductive construction.

For $a=(a_n)\in\ell^2(\mathbb N)$, define
$$
J_1a=\sum_{n=1}^{\infty}a_n\frac{x_n+y_n}{\sqrt2},
\qquad
J_2a=\sum_{n=1}^{\infty}a_n
       \frac{Tx_n+Ty_n}{\|Tx_n\|}.
$$
We first verify that these series converge in norm. Since the family $(x_1,y_1,x_2,y_2,\cdots)$ is orthonormal,
$$
\|J_1a\|=\|a\|_{\ell^2}.
$$

For the second series, we have
$$
\left\|
\sum_{n=1}^{\infty}a_n\frac{Ty_n}{\|Tx_n\|}
\right\|
\leq \sum_{n=1}^{\infty}|a_n|\eta_n
\leq
\left(\sum_{n=1}^{\infty}|a_n|^2\right)^{1/2}
\left(\sum_{n=1}^{\infty}\eta_n^2\right)^{1/2}
\leq \frac1{10}
\left(\sum_{n=1}^{\infty}|a_n|^2\right)^{1/2}.
$$
Since the vectors $Tx_n/\|Tx_n\|$ are orthonormal,
$$
\frac9{10}
\left(\sum_{n=1}^{\infty}|a_n|^2\right)^{1/2}
\leq
\left\|
\sum_{n=1}^{\infty}a_n\frac{Tx_n+Ty_n}{\|Tx_n\|}
\right\|
\leq
\frac{11}{10}
\left(\sum_{n=1}^{\infty}|a_n|^2\right)^{1/2}.
$$
Then we obtain
$$
\frac9{10}\|a\|_{\ell^2}
\leq \|J_2a\|
\leq \frac{11}{10}\|a\|_{\ell^2}.
$$
Thus $J_1,J_2:\ell^2(\mathbb N)\to\mathcal H$
are bounded linear operators and both $\operatorname{ran}J_1$ and $\operatorname{ran}J_2$ are closed. 

Define
$$
\mathcal L=\operatorname{ran}J_1=\overline{\operatorname{span}}\{x_n+y_n:n\geq1\},
\quad
\mathcal W=\operatorname{ran}J_2=\overline{\operatorname{span}}\{Tx_n+Ty_n:n\geq1\}.
$$
Then
$$
T\mathcal L\subset\mathcal W\quad
{\rm and}
\quad \mathcal L,\mathcal W\subset\mathcal M^\perp.
$$
 Since $J_1$ and $J_2$ are injective, $\mathcal L$ and $\mathcal W$ are infinite-dimensional.


Let $P$ be the orthogonal projection onto
$$
\overline{\operatorname{span}}\{y_n:n\geq1\}.
$$
For $\ell=J_1b\in\mathcal L$, we have
$$
P\ell=\frac1{\sqrt2}\sum_{n=1}^{\infty}b_ny_n,
\qquad
\|P\ell\|=\frac1{\sqrt2}\|\ell\|\geq \frac{2}{3}\|\ell\|.
$$
On the other hand, $\langle Tx_n, y_j\rangle=0$ for all $n,j$.
Thus, for $w=J_2a\in\mathcal W$,
$$
\|Pw\|
=\left\|
P\sum_{n=1}^{\infty}a_n\frac{Ty_n}{\|Tx_n\|}
\right\|
\leq\frac1{10}\|a\|_{\ell^2}
\leq\frac19\|w\|.
$$
Suppose that $\ell=w\in\mathcal L\cap\mathcal W$, then these estimates give
$$
\frac{2}{3}\|\ell\|\leq\|P\ell\|=\|Pw\|\leq\frac19\|w\|,
$$
so $\mathcal L\cap\mathcal W=\{0\}$.

Moreover, for $\ell\in\mathcal L$ and $w\in\mathcal W$,
$$
\|P\ell\|\leq\|P(\ell+w)\|+\|Pw\|\leq\|\ell+w\|+\frac19\|w\|.
$$
Together with $\|w\|\leq \|\ell+w\|+\|\ell\|$, we have
$$
\frac{2}{3}\|\ell\|\leq\|P\ell\|\leq\frac{10}{9}\|\ell+w\|+\frac19\|\ell\|.
$$
Thus
$$
\|\ell\|\leq2\|\ell+w\|.
$$
This also shows that $\mathcal L+\mathcal W$ is closed.

Set $\mathcal N=(\mathcal L+\mathcal W)^\perp$.
Every vector in $\mathcal H$ has a unique decomposition
$\ell+w+z$, where
$\ell\in\mathcal L$, $w\in\mathcal W$, and $z\in\mathcal N$.
Define
$$
Q(\ell+w+z)=\ell.
$$
Then
$$
\|Q(\ell+w+z)\|
=\|\ell\|
\leq2\|\ell+w\|
\leq2\|\ell+w+z\|.
$$
Thus $Q$ is a bounded idempotent with $\|Q\|\leq2$, and
$$
\operatorname{ran}Q=\mathcal L,
\qquad
\ker Q=\mathcal W\oplus\mathcal N.
$$
Since $\mathcal L,\mathcal W\subset\mathcal M^\perp$,
we have $\mathcal M\subset\mathcal N$.
Hence $\mathcal M,\mathcal W\subset\ker Q$.
The inclusion $T\mathcal L\subset\mathcal W$ gives
$QTQ=0$. Both the range and kernel of $Q$ are
infinite dimensional.

Relative to the bounded direct-sum decomposition
$$
\mathcal H=\mathcal L\dotplus\mathcal W\dotplus\mathcal N,
$$
the operators have the forms
$$
Q=
\begin{array}{c@{\;}c}
\begin{pmatrix}
I&0&0\\
0&0&0\\
0&0&0
\end{pmatrix}
&
\begin{matrix}
\mathcal L\\
\mathcal W\\
\mathcal N
\end{matrix}
\end{array},
\qquad
T=
\begin{array}{c@{\;}c}
\begin{pmatrix}
0&*&*\\
*&*&*\\
0&*&*
\end{pmatrix}
&
\begin{matrix}
\mathcal L\\
\mathcal W\\
\mathcal N
\end{matrix}
\end{array}.
$$
\end{proof}

For  trace-class operators,
we obtain a uniform bound from the similarity. This specific version is essential for the subsequent arguments.

\begin{lemma}
\label{lem:finite-trace-block}
Let $T\in\SSp_1(\HH)$ have infinite rank. There are an
orthogonal decomposition $\HH=\HH_0\oplus\HH_1$ into
infinite-dimensional subspaces and a bounded invertible
operator $S$ such that
\[
\|S\|,\|S^{-1}\|\leq2,\qquad
S^{-1}TS=
\begin{pmatrix}
D&E\\
F&G
\end{pmatrix}
\begin{array}{l}
\HH_0\\
\HH_1
\end{array},
\]
where $D$ has finite rank, $\Tr(D)=\Tr(T)$, $E,F$ are
compact, and $G\in\SSp_1(\HH_1)$ satisfies $\Tr(G)=0$.
\end{lemma}

\begin{proof}
Since $T\ne0$, polarization yields a unit vector $u$ with
$\ip{Tu}{u}\ne0$. By compactness, choose a unit vector
$v\in\spanop\{u,Tu,T^*u\}^{\perp}$ such that
\[
|\ip{Tv}{v}|\leq \frac{|\ip{Tu}{u}|}{9}.
\]
 Set
\[
x=\frac{u+v}{\sqrt2},\qquad
y=\frac{v-u}{\sqrt2}.
\]
Then $x,y$ are orthonormal, and
\[
\ip{Ty}{x}=\frac{\ip{Tv}{v}-\ip{Tu}{u}}{2}\ne0,\qquad
\ip{Ty}{y}=\frac{\ip{Tv}{v}+\ip{Tu}{u}}{2}.
\]
In particular,
$$
\frac{|\ip{Ty}{y}|}{|\ip{Ty}{x}|}\leq
\frac{|\ip{Tv}{v}|+|\ip{Tu}{u}|}{|\ip{Tu}{u}|-|\ip{Tv}{v}|}\leq \frac{5}{4}.
$$

Next we choose an orthonormal basis of $y^\perp$ beginning with $x$. Since the diagonal series of $T$ converges absolutely to $\Tr(T)$, there is an orthogonal projection $P$ onto a finite span of this basis such that
\[
|\Tr(T)-\Tr(TP)-\ip{Ty}{y}|<\frac{1}{4}|\ip{Ty}{x}|.
\]
Set
\[
\alpha=\frac{\Tr(T)-\Tr(TP)}{\ip{Ty}{x}}.
\]
Then we have $|\alpha|<3/2$.

Define
$$
K_0z=\langle z,x\rangle y,
\qquad
R=P+\alpha K_0.
$$
Since $Px=x$ and $Py=0$, we have
$$
PK_0=0,\qquad K_0P=K_0,\qquad K_0^2=0.
$$
Thus $R$ is a finite-rank idempotent. Moreover,
$$
R^*R
=P+|\alpha|^2K_0^*K_0
=(P-K_0^*K_0)+(1+|\alpha|^2)K_0^*K_0.
$$
Then
$$
\|R\|^2=1+|\alpha|^2<4.
$$
By the definition of $\alpha$,
$$
\Tr(TR)
=\Tr(TP)+\alpha\langle Ty,x\rangle
=\Tr(T).
$$

Apply Lemma~\ref{lem:compression} to $T$ and the
finite-dimensional space
$$
\mathcal M
=\operatorname{ran}R+\operatorname{ran}R^*
 +T(\operatorname{ran}R)+T^*(\operatorname{ran}R^*).
$$
Let $Q,\mathcal L,\mathcal W$ be the resulting objects.
Then
$$
T\mathcal L\subset\mathcal W,\qquad
\operatorname{ran}Q=\mathcal L\subset\mathcal M^\perp,
\qquad
\mathcal M,\mathcal W\subset\ker Q,
\qquad
\|Q\|\leq2.
$$
The choice of $\mathcal M$ gives
$$
RQ=QR=RTQ=QTR=0.
$$
Together with $QTQ=0$, these identities imply
$$
(R+Q)^2=R+Q,
\qquad
(R+Q)T(R+Q)=RTR.
$$

The commuting idempotents $R$ and $Q$ give the bounded
direct-sum decomposition
$$
\mathcal H
=\operatorname{ran}R
 \dotplus\mathcal L
 \dotplus\ker(R+Q).
$$
Relative to this decomposition, $T$ has the form
$$
T=
\begin{array}{c@{\;}c}
\begin{pmatrix}
*&0&*\\
0&0&*\\
*&*&*
\end{pmatrix}
&
\begin{matrix}
\operatorname{ran}R\\
\mathcal L\\
\ker(R+Q)
\end{matrix}
\end{array}.
$$

Relative to the orthogonal decomposition
$\mathcal H=\mathcal M\oplus\mathcal M^\perp$,
the operators $R$ and $Q$ act on different summands.
Therefore
$$
\|R+Q\|=\max\{\|R\|,\|Q\|\}\leq2.
$$
Set
$$
\mathcal H_0=\operatorname{ran}(R+Q),
\qquad
\mathcal H_1=\mathcal H_0^\perp.
$$
Since $\mathcal L\subset\mathcal H_0$, the space
$\mathcal H_0$ is infinite dimensional.
Moreover, $\mathcal W\subset\ker(R+Q)$.
The map
$$
\mathcal H_1\longrightarrow\ker(R+Q),
\qquad h\longmapsto (I-R-Q)h,
$$
is a bounded isomorphism, whose inverse is the restriction
of the orthogonal projection onto $\mathcal H_1$.
Thus $\mathcal H_1$ is also infinite dimensional.

Since $R+Q$ is the identity on its range, it has the form
$$
R+Q=
\begin{array}{c@{\;}c}
\begin{pmatrix}
I&R_0\\
0&0
\end{pmatrix}
&
\begin{matrix}
\mathcal H_0\\
\mathcal H_1
\end{matrix}
\end{array}.
$$
 By the spectral mapping theorem and spectral radius formula,
$$
1+\|R_0R_0^*\|
=\|I+R_0R_0^*\|
=\|R+Q\|^2
\leq4.
$$
In particular, $\|R_0\|^2=\|R_0R_0^*\|\leq3$.

Define
$$
S=\sqrt2
\begin{pmatrix}
I&-R_0/2\\
0&I/2
\end{pmatrix},
\qquad
S^{-1}=\frac1{\sqrt2}
\begin{pmatrix}
I&R_0\\
0&2I
\end{pmatrix}.
$$
Then we have
$$
S^{-1}(R+Q)S=
\begin{pmatrix}
I&0\\
0&0
\end{pmatrix}.
$$
Now we estimate the norms,
$$
\|S^*S\|=\left\|
\begin{pmatrix}
2I&-R_0\\
-R_0^*&(I+R_0^*R_0)/2
\end{pmatrix}\right\|\leq
\left\|
\begin{pmatrix}
2I&0\\
0&(I+R_0^*R_0)/2
\end{pmatrix}\right\|+\left\|
\begin{pmatrix}
0&R_0\\
R_0^*&0
\end{pmatrix}\right\|
\leq 4.
$$
Thus $\|S\|\leq 2$. A similar calculation shows $\|S^{-1}\|\leq 2$.

Write
$$
\widehat T=S^{-1}TS=
\begin{array}{c@{\;}c}
\begin{pmatrix}
D&E\\
F&G
\end{pmatrix}
&
\begin{matrix}
\mathcal H_0\\
\mathcal H_1
\end{matrix}
\end{array}.
$$
Then
$$
\begin{pmatrix}
D&0\\
0&0
\end{pmatrix}
=S^{-1}(R+Q)T(R+Q)S
=S^{-1}RTRS.
$$
Thus $D$ has finite rank, and
$$
\Tr(D)=\Tr(RTR)=\Tr(TR)=\Tr(T).
$$
Since $T$ is trace class, so is $\widehat T$.
Its blocks $E,F$ are compact, and $G$ is trace class.
Finally,
$$
\Tr(G)=\Tr(\widehat T)-\Tr(D)
=\Tr(T)-\Tr(T)=0.
$$
\end{proof}
The following lemma extends Marcoux's technique \cite[Lemma~2.2]{Marcoux} to develop a block-diagonal spectral separation technique that incorporates a prescribed compact commutator and a zero-diagonal compact operator in the diagonal entries. Moreover, our construction guarantees that the norms of the new factors are controlled by the initial operators.

\begin{lemma}\label{lem:joining}
Let $\HH_0,\HH_1$ be separable infinite-dimensional Hilbert spaces,
and let
\[
T=\begin{pmatrix}D&E\\F&G\end{pmatrix}
\in\KK(\HH_0\oplus\HH_1).
\]
Suppose that $D=[A_0,B_0]$ with $A_0,B_0\in\KK(\HH_0)$,
and that $G$ has a  zero-diagonal. Then there exist compact operators $A,B\in\KK(\HH_0\oplus\HH_1)$ such that
$$
T=[A,B],\qquad
\|A\|\|B\|
\leq 24\max\{\|A_0\|^2,\|B_0\|^2,\kappa\|T\|\}.
$$
\end{lemma}

\begin{proof}
Assume that $T\neq 0$,  and set
\[
c=\max\{\|A_0\|,\|B_0\|,\sqrt{\kappa\|T\|}\},
\qquad \delta_1=2c.
\]
Fix an orthonormal basis $(e_j)_{j\geq 1}$ of $\HH_1$ in which $G$ has a zero-diagonal.

 Since $\sigma (A_0)\backslash\{0\}$ is countable, we choose a sequence $(\delta_n)$ recursively such that
$$0<\delta_{n+1}\leq\frac{\delta_n}{4}, \qquad2\delta_{n+1}\notin \sigma(A_0).$$
Set $r_n=\|(A_0-2\delta_nI)^{-1}\|$. Since $\delta_1=2c$ and $\|A_0\|\leq c$,
\[
r_1=\|(A_0-2\delta_1I)^{-1}\|\leq\frac1{2\delta_1-\|A_0\|}
    \leq\frac1{3c}.
\]
Consequently,
\[
\|G\|\leq\|T\|\leq\frac{c^2}{\kappa}
\leq\frac1{4\kappa r_1^2},
\]
and
\[
\max\{\|E\|,\|F\|\}
\leq\|T\|\leq \frac{c^2}{\kappa}\leq\frac{c}{3\kappa r_1}\leq \frac{c}{4r_1}.
\]

Set $P_0=0$. By compactness, we can choose integers $0<d_1<d_2<\cdots$ such that the
projections $(P_n)_{n\geq 1}$ onto $\spanop\{e_1,\ldots,e_{d_n}\}$ satisfy
$$
\begin{aligned}
\max\{\|(I-P_{n-1})G\|,\|G(I-P_{n-1})\|\}
&\leq\min\left\{\frac{\delta_n^2}{\kappa},
                         \frac1{4\kappa r_n^2}\right\},\\
\max\{\|E(I-P_{n-1})\|,\|(I-P_{n-1})F\|\}
&\leq\frac{c}{2^{n+1}r_n}.
\end{aligned}
$$

Set $$Q_n=P_n-P_{n-1},\quad \mathcal E_n=Q_n\HH_1,\quad
G_{mn}=Q_mG|_{\mathcal E_n}.$$
Then
$$\|G_{mn}\|\leq \min\{\frac{\delta_m^2}{\kappa}, \frac1{4\kappa r_m^2},\frac{\delta_n^2}{\kappa},
                         \frac1{4\kappa r_n^2}\}.$$
By Lemma~\ref{lem:matrix}, we get
\[
G_{nn}=[U_n,V_n],\qquad
\max\{\|U_n\|,\|V_n\|\}\leq\sqrt{\kappa\|G_{nn}\|}.
\]
Hence for $n\geq 1$,
\[
\|U_n\|,\|V_n\|\leq \delta_n,\qquad \|U_n\|\leq \sqrt{\frac{1}{4r_n^2}}=\frac{1}{2 r_n}.
\]


Set
\[
A_n=2\delta_n I+U_n.
\]
Apply the construction in
Proposition~\ref{prop:zero-diagonal} to $G$, using these
particular diagonal factors.
This yields  compact operators on $\mathcal H_1$ of the form
\[
A_{\mathrm t}=
\begin{pmatrix}
A_1&0&0&\cdots\\
0&A_2&0&\cdots\\
0&0&A_3&\cdots\\
\vdots&\vdots&\vdots&\ddots
\end{pmatrix},
\qquad
B_{\mathrm t}=
\begin{pmatrix}
V_1&B_{12}&B_{13}&\cdots\\
B_{21}&V_2&B_{23}&\cdots\\
B_{31}&B_{32}&V_3&\cdots\\
\vdots&\vdots&\vdots&\ddots
\end{pmatrix},
\]
such that
\[
[A_{\mathrm t},B_{\mathrm t}]=G,\quad
\|A_{\mathrm t}\|\leq 3\delta_1,\quad\|B_{\mathrm t}\|\leq
 \frac{4}{3}\delta_1.
\]

Applying Lemma~\ref{lem:resolvent} with
$$
E_n:=E|_{\mathcal E_n}\in \mathcal{B}(\mathcal{E}_n,\mathcal{H}_0), \quad F_n:=Q_nF\in  \mathcal{B}(\mathcal{H}_0,\mathcal{E}_n),\quad
A_n=2\delta_n I+U_n,$$
gives operators
$X_n:\mathcal E_n\to\HH_0$ and $Y_n:\HH_0\to\mathcal E_n$
such that
\[
A_0X_n-X_nA_n=E_n,\qquad A_nY_n-Y_nA_0=F_n,
\]
and for $n\geq 1$,
\[
\|X_n\|\leq2r_n\|E_n\|\leq 2r_n\|E(I-P_{n-1})\|\leq \frac{c}{2^{n}},
\quad
\|Y_n\|\leq2r_n\|F_n\|\leq\frac{c}{2^{n}}.
\]

Define the row and column operators
\[
X=(X_1\,\,X_2\,\,X_3\,\,\cdots)
:\mathcal H_1\to\mathcal H_0,
\quad
Y=\begin{pmatrix}Y_1\\[1mm]
Y_2\\[1mm]
Y_3\\[1mm]
\vdots\end{pmatrix}
:\mathcal H_0\to\mathcal H_1.
\]
Thus, for $x=(x_n)\in\bigoplus_{n\geq1}\mathcal E_n$
and $h\in\mathcal H_0$,
\[
Xx=\sum_{n=1}^{\infty}X_nx_n,
\qquad
Yh=(Y_nh)_{n\geq1}.
\]
Then they are bounded operators with $\|X\|,\|Y\|\leq c$. The compactness of
 $X$ and $Y$ follows from the fact that
 $$
 \max\{\|X-XP_k\|,\,\|Y-P_kY\|\}\leq \frac{c}{2^k}\to 0.
 $$

Now we sum the finite partial terms of the following equations
$$
A_0X_n-X_nA_n=E_n,\qquad A_nY_n-Y_nA_0=F_n.
$$
Then we get
\[
A_0XP_k-XA_{\mathrm t}P_k=EP_k,\qquad P_kA_{\rm t}Y-P_kYA_0=P_kF.
\]
Since $A_{\mathrm t}$ commutes with $P_k$, and $EP_k\to E$, $P_kF\to F$, taking norm limits gives
\[
A_0X-XA_{\mathrm t}=E,\qquad
A_{\mathrm t}Y-YA_0=F.
\]

Define
\[
A=\begin{pmatrix}A_0&0\\0&A_{\mathrm t}\end{pmatrix},\qquad
B=\begin{pmatrix}B_0&X\\Y&B_{\mathrm t}\end{pmatrix}
\begin{array}{l}\HH_0\\\HH_1\end{array}.
\]
Both operators are compact and
\[
[A,B]
=
\begin{pmatrix}
[A_0,B_0]&A_0X-XA_{\mathrm t}\\
A_{\mathrm t}Y-YA_0&[A_{\mathrm t},B_{\mathrm t}]
\end{pmatrix}
=
\begin{pmatrix}
D&E\\
F&G
\end{pmatrix}.
\]
The preceding estimates yield
$$
\|A\|\leq \max\{\|A_0\|,\|A_{\rm t}\|\}\leq 3\delta_1=6c,
$$
and
$$
\|B\|=\left\|\begin{pmatrix}B_0&X\\Y&B_{\mathrm t}\end{pmatrix}\right\|\leq \left\|\begin{pmatrix}B_0&0\\0&B_{\mathrm t}\end{pmatrix}\right\|
+\left\|\begin{pmatrix}0&X\\Y&0\end{pmatrix}\right\|\leq \frac{4}{3}\delta_1+c\leq 4c.
$$
Then  $\|A\|\|B\|\leq 24c^2$. 
\end{proof}
\section{Main Results}

We first treat compact operators outside the trace class and then
consider the trace-class case. 
We recall Fan's criterion for the existence of a zero-diagonal.
For a detailed proof, see
\cite[Appendix~B.1]{LoreauxThesis}.

\begin{theorem}[{\cite[Theorem~1]{Fan}}]\label{thm:fan}
An operator $T\in\BB(\HH)$ has a zero-diagonal if and only if
there are an orthonormal basis $(e_j)_{j\geq1}$ and a strictly increasing sequence of integers
$(m_k)_{k\geq 1}$ such that
\[
\lim_{k\to\infty}\sum_{j=1}^{m_k}\ip{Te_j}{e_j}=0.
\]
In particular, every $T\in\SSp_1(\HH)$ with $\Tr(T)=0$
has  a zero-diagonal.
\end{theorem}


For every $T\in\KK(\HH)\setminus\SSp_1(\HH)$, Fan and Fong \cite{FF}
proved that $T$ is similar to an operator with zero-diagonal.
We prove the following quantitative version. 

\begin{theorem}\label{thm:bounded-similarity}
Let $T\in\KK(\HH)\setminus\SSp_1(\HH)$ and $0<\theta<1$. There is a bounded
invertible operator $S$ such that
\[
\|S\|,\|S^{-1}\|\leq 1+2\theta
\]
and $S^{-1}TS$ has a zero-diagonal.
\end{theorem}

\begin{proof}
Fix an orthonormal basis $(h_j)_{j\geq1}$ of $\HH$.
Since $T$ and $T^*$ are compact,
$\|Th_j\|+\|T^*h_j\|\to0$.
We may choose a sequence of numbers $(j_n)$ satisfying $j_{n+1} > j_n + 1$ such that
$$\Vert{}Th_{j_n}\Vert{} + \Vert{}T^*h_{j_n}\Vert{} \leq \frac{1}{2^n}.$$
Let
$$y_n = h_{j_n},\quad \mathcal N=\overline{\spanop\{y_n:n\geq1\}},\quad \mathcal M=\mathcal N^\perp.$$
By the choice of $j_n$, it follows that $\mathcal M$
and $\mathcal N$ are infinite-dimensional. Let $P$ be the orthogonal projection onto $\mathcal N$. Then the series
\[
TP=\sum_{n=1}^{\infty}\ip{\,\cdot\,}{y_n}Ty_n,\qquad
PT=\sum_{n=1}^{\infty}\ip{\,\cdot\,}{T^*y_n}y_n
\]
converge in trace norm. 
Define
\[
C=(I-P)T|_{\mathcal M},\qquad K=PT+TP-PTP.
\]
Then $K\in\SSp_1(\HH)$ and
$$
T=(C\oplus0)+K
$$
relative to $\HH=\mathcal M\oplus\mathcal N$.
Since $T\notin\SSp_1(\HH)$, it follows that
$C\notin\SSp_1(\mathcal M)$.

Write $C=\operatorname{Re}C+i\operatorname{Im}C$. At least one of its real and imaginary parts
is not trace class. 
By replacing $T$ with $-iT$ if necessary, we can assume that $\operatorname{Re}C$ is not trace class. By the spectral theorem,
there is an orthonormal basis $(x_n)_{n\geq1}$ of
$\mathcal M$ such that
\[
(\operatorname{Re}C)x_n=\lambda_nx_n,
\qquad \lambda_n\in\mathbb R.
\]
Then  $\sum_{n=1}^{\infty}|\lambda_n|=\infty$.
Set $a_n=\ip{Cx_n}{x_n}$, then
$\operatorname{Re}a_n=\lambda_n$, and hence
\[
\sum_{n=1}^{\infty}|a_n|
\geq\sum_{n=1}^{\infty}|\lambda_n|
=\infty.
\]
The compactness of $C$ implies $a_n\to0$.

Now we partition
$\mathbb N$ into four index sets $I_1,I_{-1},I_i,I_{-i}$ such that
\[
\sum_{n\in I_\omega}|a_n|=\infty,
\qquad\omega\in\{1,-1,i,-i\}.
\]
For $n\in I_\omega$, set
\[
q_n=
\begin{cases}
\displaystyle \frac{\omega|a_n|}{a_n},&a_n\ne0,\\[6pt]
0,&a_n=0.
\end{cases}
\]
Then $|q_n|\leq1$ and $q_na_n=\omega|a_n|$.
Let
\[
\mathcal F_n=\operatorname{span}\{x_n,y_n\}.
\]
Since $(x_n)$ and $(y_n)$ are orthonormal bases of
$\mathcal M$ and $\mathcal N$, respectively,
\[
\mathcal H=\bigoplus_{n=1}^{\infty}\mathcal F_n.
\]
Relative to the ordered basis $(x_n,y_n)$ of $\mathcal F_n$,
define
\[
S_n=
\begin{pmatrix}
1&\theta\\
-\theta q_n&1-\theta^2 q_n
\end{pmatrix},
\quad
S_n^{-1}=
\begin{pmatrix}
1-\theta^2 q_n&-\theta\\
\theta q_n&1
\end{pmatrix}.
\]

A straightforward computation shows that
\[
\max\{\|S_n-I\|,\|S_n^{-1}-I\|\}
\leq \theta+\theta^2
\leq 2\theta.
\]
Let
$
S=\bigoplus_{n=1}^{\infty}S_n.
$
Then $S$ is bounded and invertible, with
\[
S^{-1}=\bigoplus_{n=1}^{\infty}S_n^{-1},
\qquad
\|S\|,\ \|S^{-1}\|\leq 1+2\theta.
\]

Recall that $T=(C\oplus0)+K$, where $K$ is trace class.
Set
$$
\widehat T=S^{-1}TS,
\qquad
\gamma_n=\langle S^{-1}KSy_n,y_n\rangle.
$$
Since $S=\bigoplus_n S_n$, the compression of
$S^{-1}(C\oplus0)S$ to $\mathcal F_n$, relative to
$(x_n,y_n)$, is
$$
S_n^{-1}
\begin{pmatrix}
a_n&0\\
0&0
\end{pmatrix}
S_n
=
a_n
\begin{pmatrix}
1-\theta^2q_n&\theta(1-\theta^2q_n)\\
\theta q_n&\theta^2q_n
\end{pmatrix}.
$$
Taking the $(2,2)$-entry and using
$q_na_n=\omega|a_n|$ gives
$$
\langle\widehat Ty_n,y_n\rangle
=\theta^2q_na_n+\gamma_n
=\theta^2\omega|a_n|+\gamma_n,
\qquad n\in I_\omega.
$$
Moreover, $S^{-1}KS$ is trace class, and hence
\[
\sum_{n=1}^{\infty}|\gamma_n|
\leq\|S^{-1}KS\|_1
\leq(1+2\theta)^2\|K\|_1<\infty.
\]
Let
$
e_{2n-1}=x_n,\,\,e_{2n}=y_n.
$
Now we construct finite sets inductively
\[
\varnothing=F_0\subsetneq F_1\subsetneq F_2\subsetneq\cdots
\]
such that
\[
\bigcup_{k=1}^{\infty}F_k=\mathbb N,
\qquad
\left|
\sum_{j\in F_k}\langle\widehat Te_j,e_j\rangle
\right|\leq \frac{3}{2^k}.
\]

Suppose that $F_{k-1}$ has been constructed.
Choose $N_k\geq k$ such that
\[
F_{k-1}\subset\{1,\ldots,2N_k-1\},
\qquad
\sup_{n>N_k}\theta^2|a_n|\leq\frac{1}{2^k},
\qquad
\sum_{n>N_k}|\gamma_n|\leq\frac{1}{2^k}.
\]
Denote
\[
z=\sum_{j=1}^{2N_k}\langle\widehat Te_j,e_j\rangle.
\]
Suppose that $\operatorname{Re}z\geq0$ and
$\operatorname{Im}z\geq0$.
Let $p,q\geq N_k$ be the least integers satisfying
$$
\sum_{\substack{N_k<n\leq p\\n\in I_{-1}}}\theta^2|a_n|
\geq\operatorname{Re}z,
\qquad
\sum_{\substack{N_k<n\leq q\\n\in I_{-i}}}\theta^2|a_n|
\geq\operatorname{Im}z.
$$
These integers exist because both sums diverge.
Set
$$
J=\{n\in I_{-1}:N_k<n\leq p\},
\qquad
L=\{n\in I_{-i}:N_k<n\leq q\}.
$$
Hence
$$
0\leq\theta^2\sum_{n\in J}|a_n|
       -\operatorname{Re}z\leq\frac{1}{2^k},\quad
0\leq\theta^2\sum_{n\in L}|a_n|
       -\operatorname{Im}z\leq\frac{1}{2^k}.
$$

Define
$$
F_k=\{1,\ldots,2N_k\}\cup\{2n:n\in J\cup L\}.
$$
Using
$\langle\widehat Ty_n,y_n\rangle
=\theta^2\omega|a_n|+\gamma_n$,
with $\omega=-1$ on $J$ and $\omega=-i$ on $L$, we obtain
$$
\begin{aligned}
\sum_{j\in F_k}\langle\widehat Te_j,e_j\rangle
&=z+\sum_{n\in J}\langle\widehat Ty_n,y_n\rangle
    +\sum_{n\in L}\langle\widehat Ty_n,y_n\rangle\\
&=z+\sum_{n\in J}(-\theta^2|a_n|+\gamma_n)
    +\sum_{n\in L}(-i\theta^2|a_n|+\gamma_n)\\
&=z-\theta^2\sum_{n\in J}|a_n|
    -i\theta^2\sum_{n\in L}|a_n|
    +\sum_{n\in J\cup L}\gamma_n.
\end{aligned}
$$
Consequently,
$$
\left|\sum_{j\in F_k}\langle\widehat Te_j,e_j\rangle\right|
\leq
\left|\operatorname{Re}z-\theta^2\sum_{n\in J}|a_n|\right|
+\left|\operatorname{Im}z-\theta^2\sum_{n\in L}|a_n|\right|
+\sum_{n\in J\cup L}|\gamma_n|
\leq \frac{3}{2^k}.
$$
The other sign choices are treated similarly, using $I_1$
when $\operatorname{Re}z<0$ and $I_i$ when
$\operatorname{Im}z<0$.
Also,
$$F_{k-1}\subsetneq F_k,
\qquad\{1,\ldots,2k\}\subset F_k.$$
This completes the induction.

Now we order the indices by sequentially listing the elements of $F_1, F_2 \setminus F_1, F_3 \setminus F_2, \ldots$, and let $(f_j)$ be the corresponding orthonormal basis. Set $m_k=|F_k|$,
then
\[
\sum_{j=1}^{m_k}\langle\widehat Tf_j,f_j\rangle
=
\sum_{j\in F_k}\langle\widehat Te_j,e_j\rangle
\longrightarrow0.
\]
Now Theorem~\ref{thm:fan} implies that $\widehat T$
has a zero-diagonal.
\end{proof}

\begin{proposition}\label{prop:beyond-trace-class}
For every $T\in\KK(\HH)\setminus\SSp_1(\HH)$, there are
$A,B\in\KK(\HH)$ such that
\[
T=[A,B],\qquad \|A\|\,\|B\|\leq 2^{5}\kappa\|T\|.
\]
\end{proposition}

\begin{proof}
Let $\theta=1/6$, choose $S$ as in Theorem~\ref{thm:bounded-similarity}, then $\|S\|\|S^{-1}\|\leq 2$.  Set
$\widehat T=S^{-1}TS$.
Proposition~\ref{prop:zero-diagonal} gives compact operators
$\widehat A,\widehat B$ with
\[
\widehat T=[\widehat A,\widehat B],\qquad
\|\widehat A\|\,\|\widehat B\|
\leq4\kappa\|\widehat T\|.
\]
Set $A=S\widehat AS^{-1}$ and $B=S\widehat BS^{-1}$.
These operators are compact, and
$[A,B]=S[\widehat A,\widehat B]S^{-1}=T$.
Furthermore,
\[
\|A\|\,\|B\|
\leq 4\|\widehat A\|\,\|\widehat B\|
\leq 16\kappa\|\widehat T\|
\leq 2^{5}\kappa\|T\|.
\]
\end{proof}
We now turn to the trace-class case. Anderson's rank-one
construction provides the finite-rank factors needed for the
first diagonal block in Lemma~\ref{lem:finite-trace-block}.

\begin{lemma}\label{lem:finite-rank}
Let $T\in\KK(\HH)$ have finite rank, then there exist compact operators  $A_0,B_0$  such that
$$
T=[A_0,B_0],\qquad \max\{\|A_0\|,\|B_0\|\}\leq 2\|T\|^{1/2}.
$$
\end{lemma}

\begin{proof}
Assume $T\neq 0$. Let $h$ be a fixed unit vector. By
Anderson's construction \cite{Anderson}, which is reproduced in
\cite[proof of Theorem~3.1]{BPW}, there are  compact operators
$C_0,Z_0$ on $\ell^2(\mathbb N)$ such that
\[
[C_0,Z_0]=P_h,
\qquad
\max\{\|C_0\|,\|Z_0\|\}\leq2.
\]

Set
$$
\mathcal E=\operatorname{ran}T+\operatorname{ran}T^*,
\qquad
T_{\mathcal E}=T|_{\mathcal E}.
$$
Then $\mathcal E$ is finite dimensional and reduces $T$,
with $T=T_{\mathcal E}\oplus0$ on
$\mathcal E\oplus\mathcal E^\perp$.

On $\mathcal E\otimes\ell^2(\N)$, define
$$
C=\|T\|^{-1/2}T_{\mathcal E}\otimes C_0,
\qquad
Z=\|T\|^{1/2}I_{\mathcal E}\otimes Z_0.
$$
Since $\mathcal E$ is finite-dimensional, $C,Z$ are compact and
\[
[C,Z]
=T_{\mathcal E}\otimes[C_0,Z_0]
=T_{\mathcal E}\otimes P_h,\qquad
\|C\|,
\|Z\|\leq 2 \|T\|^{1/2}.
\]
The operator $T_{\mathcal E}\otimes P_h$ is unitarily
equivalent to $T$. Let $U$ be such a unitary. 
Set
\[
A_0=UCU^*,\qquad B_0=UZU^*.
\]
Thus
\[
[A_0,B_0]=U[C,Z]U^*=T.
\]
Moreover,
$$
\|A_0\|,\|B_0\|\leq \max\{\|C_0\|,\|Z_0\|\}\cdot\|T\|^{1/2}\leq 2\|T\|^{1/2}.
$$
\end{proof}

\begin{proposition}\label{prop:trace-class}
For every $T\in\SSp_1(\HH)$, there are $A,B\in\KK(\HH)$
such that
\[
T=[A,B],\qquad \|A\|\,\|B\|\leq 2^{11}\kappa\|T\|.
\]
\end{proposition}

\begin{proof}
If $T$ has finite rank, apply Lemma~\ref{lem:finite-rank}, there exist $A, B$ such that
$$
\|A\|\|B\|\leq4\|T\|.
$$
If $\Tr(T)=0$, Theorem~\ref{thm:fan} gives a zero diagonal,
and Proposition~\ref{prop:zero-diagonal} applies, then
\[
\|A\|\,\|B\|\leq4\kappa\|T\|.
\]
Suppose that $T$ has infinite rank and $\Tr(T)\ne0$.
By Lemma~\ref{lem:finite-trace-block}, there are a bounded
invertible $S$ and an orthogonal decomposition into
infinite-dimensional subspaces such that
\[
\|S\|,\|S^{-1}\|\leq 2,\qquad
\widehat T=S^{-1}TS
=\begin{pmatrix}D&E\\F&G\end{pmatrix}
\begin{array}{l}\HH_0\\\HH_1\end{array},
\]
where $D$ has finite rank and $G$ is trace class with
$\Tr(G)=0$.
Theorem~\ref{thm:fan} therefore gives an orthonormal basis of
$\HH_1$ in which $G$ has zero diagonal.
Applying Lemma~\ref{lem:finite-rank} on $\HH_0$ gives
\[
D=[A_0,B_0],\qquad
\max\{\|A_0\|,\|B_0\|\}
\leq 2\|D\|^{1/2}\leq 2\|\widehat T\|^{1/2}.
\]
By Lemma~\ref{lem:joining}, there exist compact operators $ \widehat A,\widehat B $ such that $[\widehat A,\widehat B]=\widehat T$ and
\[
\|\widehat A\|\|\widehat B\|
\leq 24\max\{\|A_0\|^2,\|B_0\|^2, \kappa\|\widehat T\|\}
\leq 24\kappa\|\widehat T\|.
\]
Here we use $\kappa\geq 4$. Finally, set $A=S\widehat AS^{-1}$ and $B=S\widehat BS^{-1}$.
Then $A,B$ are compact and $[A,B]=T$.
Then we obtain
\[
\|A\|\,\|B\|
\leq4^2\|\widehat A\|\,\|\widehat B\|
\leq 4^2\cdot 24\kappa\|\widehat T\|
\leq 2^{11}\kappa\|T\|.
\]
\end{proof}

\begin{proof}[Proof of Theorem~\ref{thm:main}]
Propositions~\ref{prop:beyond-trace-class} and
\ref{prop:trace-class} cover all compact operators. Then for any $T\in \mathcal{K}(\mathcal{H})$, there exist $A,B\in\KK(\HH)$ such that $T=[A,B]$ and
$$
\|A\|\|B\|\leq  2^{11}\kappa\|T\|.
$$
By replacing $A$ and $B$ with $\lambda A$ and $\lambda^{-1} B$ for a suitable scalar $\lambda$,  we can make their norms equal. This preserves the commutator. Hence, we can choose $A,B$ such that
\[
\max\{\|A\|,\|B\|\}\leq 2^{5}\cdot\sqrt{2}\sqrt{\kappa\|T\|}\leq 50\sqrt{\kappa\|T\|}.
\]
\end{proof}

\section*{Acknowledgement}
The author would like to thank Yuanhang Zhang  for helpful conversations and suggestions.
\section*{AI statement}
The author used GPT 6 astra during the development of this work for searching necessary references, examining mathematical arguments, and language editing.  The author has checked and substantially revised all mathematical arguments and takes full responsibility for the final version.

\end{document}